\documentclass[12pt]{article}

\usepackage[paper=letterpaper,margin=0.7in,twoside=false,includehead]{geometry}

\usepackage{amsfonts,amsmath,amsthm,amssymb, thmtools,thm-restate,hyperref,graphicx,fancyhdr,xparse,mathtools,xcolor,verbatim,multicol}  
\usepackage[capitalize]{cleveref}
\usepackage[shortlabels]{enumitem}
\usepackage[normalem]{ulem}
\usepackage[T1]{fontenc}

\SetEnumitemKey{thmparts}{topsep=2pt plus 1pt minus1pt, itemsep=1pt plus0.5pt minus0.5pt}
\DeclarePairedDelimiter\abs{|}{|}
\DeclarePairedDelimiter\set{\{}{\}}

\usepackage{lineno}

\newtheorem{thm}{Theorem}[section]
\newtheorem{lem}[thm]{Lemma}
\newtheorem{cor}[thm]{Corollary}
\newtheorem{conj}[thm]{Conjecture}

\theoremstyle{definition}
\newtheorem{defn}[thm]{Definition}
\newtheorem{rem}[thm]{Remark}
\newtheorem{example}[thm]{Example}
\newtheorem{prop}[thm]{Proposition}

\newcommand{\A}{\mathcal{A}}

\newcommand{\F}{\mathcal{F}}
\newcommand{\W}{\mathcal{W}}

\newcommand{\diam}{\mathord{\diamond}}

\newcommand{\p}[1]{{\textcolor{red}{#1}}}

\begin{document}

\title{Universal Partial Tori}
\author{William D. Carey, Matthew David Kearney, Rachel Kirsch, and Stefan Popescu}
\date{January 24, 2025}
\maketitle{}

\abstract{
    A De Bruijn cycle is a cyclic sequence in which every word of length $n$ over an alphabet $\A$ appears exactly once. De Bruijn tori are a two-dimensional analogue. Motivated by recent progress on universal partial cycles and words, which shorten De Bruijn cycles using a wildcard character, we introduce universal partial tori and matrices. We find them computationally and construct infinitely many of them using one-dimensional variants of universal cycles, including a new variant called a universal partial family.
}

\section{Introduction}\label{sec:intro}
    Universal cycles represent collections of combinatorial objects compactly, in a cyclic, overlapping way. For example, $0000100110101111$ is a De Bruijn cycle (a universal cycle for words) because it contains every binary word of length $4$ exactly once.
    
    In the last several years, compression of universal cycles has been explored, including universal cycles for words \cite{chen2017universal, Fillmore2023, Goeckner2018}, permutations \cite{KLSS23, KPV19}, surjective functions \cite{ChenKitaev2020UniversalPartialWords}, set partitions \cite{ChenKitaev2020UniversalPartialWords}, and graphs \cite{KSS23}. For example, $001\diam110\diam$ is a universal partial cycle (upcycle) because it covers every binary word of length $4$ exactly once, where the wildcard character $\diam$ can represent either $0$ or $1$, and it is only half the length of a De Bruijn cycle for the same set of words.

    Separately, De Bruijn cycles have been generalized to two dimensions. A De Bruijn torus for $\A^{w\times \ell}$, also known as a perfect map, is a torus (a matrix considered to wrap around both horizontally and vertically) in which each of the $w \times \ell$ matrices over the alphabet $\A$ appears exactly once. De Bruijn tori have been studied extensively \cite{FFMS85,Etzion1988,HI93,Kreitzer2024,Paterson94,RS62} and have applications to position sensing \cite{BM93,HS16,MP94}.

    Combining these two lines of research, we introduce the universal partial torus (uptorus), an analogue of the De Bruijn torus that uses a wildcard character $\diam$ and covers each $w \times \ell$ matrix over an alphabet $\A$ exactly once, and its non-cyclic form, the upmatrix. 

    The paper is organized as follows. In \cref{sec:prelim}, we define the relevant variants of universal cycles and related terms. 
    In Sections \ref{sec:results-comp-upmatrices} and \ref{sec:results-constructed-upmatrices} we explore upmatrices, identifying many examples computationally and laying out a construction that uses universal partial words to produce infinitely many of them.
    
    Uptori appear to be rarer than upmatrices. In \cref{sec:results-comp-uptori} we identify a single minimal uptorus computationally. Nevertheless, in \cref{sec:results-constructed-uptori} we provide a construction that produces infinitely many uptori. This construction takes as its input an upcycle and a De Bruijn cycle.
    
    We then generalize that construction following the pattern laid down by Kreitzer, Nica, and Pereira \cite{Kreitzer2024} for constructing De Bruijn tori. In \cref{sec:upfamilies}, we replace the upcycle with what we call a universal partial family (upfamily) and the De Bruijn cycle with an alternating De Bruijn cycle. This construction produces uptori for the same alphabet and subarray size, but with different dimensions from our first construction. Upfamilies are a variant of De Bruijn families (also called perfect factors of the De Bruijn graph \cite{Paterson1995}) in which a wildcard symbol is used. We show examples of and make a conjecture about producing upfamilies by slicing upcycles. Using such upfamilies in the uptori construction yields uptori which have the same number of $\diam$'s in each $w\times \ell$ subarray. 
    
    In \cref{sec:quasifamilies} we introduce universal partial quasi-families (UPQFs), a relaxation of universal partial families in which the cycles may have different lengths. While UPQFs cannot be used in our uptorus construction directly, in \cref{sec:liftedupfamilies} we show that UPQFs can be lifted to upfamilies which in turn can be used to construct a more irregular type of uptorus, in which some $w \times \ell$ subarrays have different numbers of $\diam$'s.

    Finally, in \cref{sec:open} we lay out some unanswered questions prompted by this research.

\section{Preliminaries}\label{sec:prelim}

    A \emph{universal word} for $\A^n$ is a word over the alphabet $\A$ that contains every word in $\A^n$ exactly once as a substring, i.e., a non-cyclic form of a De Bruijn cycle. For example, $0120221100$ is a universal word for $\set{0,1,2}^2$. Universal words exist for all $\A^n$. A \emph{partial word} over an alphabet $\A$ is a sequence of symbols over $\A\cup\set{\diam}$, where $\diam \notin \A$. The $\diam$ is thought of as a wildcard symbol which can correspond to any letter of the alphabet. We say that a partial word $u$ \emph{covers} a partial word $v$ if $v$ can be obtained from a substring of $u$ by replacing some subset of the diamonds with letters from $\A$. A \emph{universal partial word} (or \emph{upword}) for $\A^n$ is a partial word that covers every word in $\A^n$ exactly once \cite{chen2017universal}. A \emph{universal partial cycle} (or \emph{upcycle}) for $\A^n$ is a cyclic partial word which covers every word in $\A^n$ exactly once.  For example, $001\diam110\diam$ is an upcycle for $\set{0,1}^4$ because it covers every binary word of length $4$ exactly once (wrapping around because the partial word is cyclic).

    We analogously define a \emph{partial matrix} over an alphabet $\A$ as a rectangular matrix of symbols in $\A \cup \set{\diam}$, where $\diam \notin \A$. 
    A \emph{torus} over an alphabet $\A$ is a matrix of symbols in $\A$ where both the rows and columns are considered cyclically. A \emph{partial torus} over $\A$ is a torus over $\A\cup\set{\diam}$, where $\diam \notin \A$, i.e., a partial matrix considered cyclically in both directions. Within each of these objects we consider \emph{subarrays}: 

    \begin{defn}[Subarray]\label{def:subarray}
        A $w \times \ell$ \emph{subarray} $p$ of a (partial) matrix $m$ is a (partial) matrix with $w$ rows and $\ell$ columns such that there exist integers $r$ and $c$ with $p[y,x] = m[y+r,x+c]$ for every $x \in \set{0, \ldots, \ell-1}$ and $y \in \set{0, \ldots, w - 1}$. If $m$ is a (partial) torus, then we read $y+r$ and $x+c$ modularly.
    \end{defn}

    \begin{example}
    The partial matrix $p = \begin{bmatrix}
        0 & 1 & 1\\
        \diam & 1 & \diam\\
    \end{bmatrix}$ is a $2 \times 3$ subarray of 
    $m = \begin{bmatrix}
        0 & 0 & 1 & 1& 0\\
        0 & \diam & 1 & \diam & 0 \\
        1 & 0 & 0 & 1 & 1
    \end{bmatrix}.$ The top-left entry of $p$ is located at $(r,c) = (0,1)$ in $m$, and $p[y,x] = m[y,x+1]$ for every $x \in \set{0,1,2}$ and $y \in \set{0,1}$.
    \end{example}

    We write $\A^{w \times \ell}$ for the set of matrices over $\A$ with $w$ rows and $\ell$ columns. A partial matrix $m$ \emph{covers} each partial matrix which can be obtained by replacing some subset of the $\diam$'s in a subarray of $m$ with letters of $\A$.
    
    \begin{defn}[Universal Partial Matrix]\label{def:upmatrix}
        A \emph{universal partial matrix} (or \emph{upmatrix}) for $\A^{w\times \ell}$ is a partial matrix over the alphabet $\A$ in which every matrix in $\A^{w \times \ell}$ is covered exactly once.
    \end{defn}

    For example,
            \[
                \mu = \begin{bmatrix}
                    0 & 0 & 1 & 1& 0\\
                    0 & \diam & 1 & \diam & 0 \\
                    1 & 0 & 0 & 1 & 1
                \end{bmatrix}
            \]
    is a universal partial matrix for $\A = \set{0,1}^{2 \times 2}$.

    \begin{defn}[Universal Partial Torus]\label{def:uptorus}
        A \emph{universal partial torus} (or \emph{uptorus}) for $\A^{w \times \ell}$ is a partial torus over the alphabet $\A$ in which every matrix in $\A^{w \times \ell}$ is covered exactly once.
    \end{defn}
    
    Goeckner et al.\cite{Goeckner2018} define a \emph{window} of upcycle $u$ for $\A^n$ as a sequence of $n$ consecutive symbols of $u$. We similarly refer to the $w \times \ell$ subarrays of an upmatrix or uptorus as windows. 
    The \emph{diamondicity} of an upcycle or upword $u$ is the number of $\diam$ symbols that appear in each window of $u$, which is well-defined for all upcycles \cite[Lemma 14]{chen2017universal} and all upwords over non-binary alphabets \cite[Theorem 4.1]{Goeckner2018}. The upcycle $001\diam110\diam$ has diamondicity 1, for example. If each window of an upmatrix or uptorus $m$ has the same number of diamonds, then that number is the \emph{diamondicity} of $m$.

    We will exclude from our consideration upmatrices and uptori that we call \emph{trivial}. There are several categories of trivial upmatrices and uptori.

    First, we say that upmatrices and uptori with no $\diam$'s are trivial because they are De Bruijn tori. Next, an upmatrix or uptorus is trivial if all of its characters are $\diam$'s. For a given subarray size $w \times \ell$ there is exactly one all-$\diam$'s upmatrix (a $w \times \ell$ matrix of $\diam$'s) and exactly one all-$\diam$'s uptorus (a single $\diam$).

    Finally, an upmatrix or uptorus for $\A^{w\times \ell}$ is trivial if either $w=1$ or $\ell=1$. If $w=1$ and the height of the upmatrix or uptorus is also 1 (or $\ell$ and the width respectively), then the resulting object is either an upword or an upcycle, which are treated more thoroughly elsewhere. For example, $(001\diam110\diam)$ would be a trivial uptorus for $\set{0,1}^{1\times4}$. It is possible for an uptorus for $\A^{1\times \ell}$ to have a height greater than one, as described for De Bruijn tori in \cite{RS62}. For example, the torus
            \[
                m = \left[\begin{matrix}
                    0&0&1&\diam&1&1&0&\diam \\
                    0&0&3&\diam&1&1&2&\diam \\
                    0&2&1&\diam&1&3&0&\diam \\
                    0&2&3&\diam&1&3&2&\diam \\
                    2&0&1&\diam&3&1&0&\diam \\
                    2&0&3&\diam&3&1&2&\diam \\
                    2&2&1&\diam&3&3&0&\diam \\
                    2&2&3&\diam&3&3&2&\diam \\
                \end{matrix}\right]
            \]
    is a trivial uptorus for $\set{0,1,2,3}^{1\times4}$. The rows (considered cyclically) of this uptorus are the elements of a universal partial family; we discuss universal partial families in \cref{sec:upfamilies}, and \cref{ex:upfamily} corresponds to the rows of the uptorus $m$ above. 
    
    The dimensions of uptori are restricted when excluding trivial uptori.

    \begin{lem}\label{min-dimension}
        Every nontrivial uptorus for $\A^{w \times \ell}$ has at least $w+1$ rows and at least $\ell+1$ columns.
    \end{lem}

    \begin{proof}
        Suppose not, and let $t$ be a nontrivial uptorus for $\A^{w\times \ell}$ having exactly $w$ rows (the proof for an uptorus having exactly $\ell$ columns is similar). Assume without loss of generality that $0 \in \A$, and let $m$ be the $w \times \ell$ matrix of all $0$'s. By the definition of an uptorus, $t$ covers $m$. Because $t$ has exactly $w$ rows, there are $\ell$ consecutive columns of $t$ that are all $0$'s and $\diam$'s. But then $t$ covers $m$ more than once, in particular $w > 1$ times, once starting at each row of the $\ell$ columns of all $0$'s and $\diam$'s, contradicting the assumption that $t$ is an uptorus.
    \end{proof}

    The constructions we develop for uptori and upmatrices rely on notation to describe two-dimensional cyclic objects and their elements. If $c$ is a sequence of characters in $\A$ and $0 \leq n < \abs{c}$, then $c_n$ is the $n^\text{th}$ character of $c$. If $c$ is a word over $\A$ and $n$ is some integer, by $c^n$ we mean the concatenation of $n$ repeated copies of $c$. 

    Talking about ``rotations'' of cyclic objects requires some care. Consider the De Bruijn cycle for $\set{0,1}^2$. It can be represented in several equally valid ways: $0011$, $1001$, $0011$, and $0110$. Throughout the paper, when we refer to indices of a cyclic (partial) word, the proofs do not depend on any specific choice of starting character; the zero index can be selected arbitrarily or determined by an ordering such as using Lyndon words.

    For our construction, however, we need to describe relative rotations of representations of a given cyclic partial word. Consider some representation of a cyclic partial word $c = c_0 \cdots c_{\abs{c}-1}$. We define a forward rotation function $\sigma(c) = c_1 \cdots c_{\abs{c}-1}c_0$. Repeated applications of $\sigma$ are denoted $\sigma^i(c) = c_i \cdots c_{\abs{c}-1} c_0 \cdots c_{i-1}$.
    
    We define alternating words and alternating De Bruijn cycles as in \cite{Kreitzer2024}.

    \begin{defn}[Alternating Word]\label{def:aword}
        An \emph{alternating word} on an alphabet pair $(\mathcal{A}, \mathcal{B})$ is a word in $\mathcal{A} \times \mathcal{B} \times \mathcal{A} \times \cdots \times \mathcal{B} \times \mathcal{A}$. This word has length $2n + 1$, where $n$ is a non-negative integer.
    \end{defn}

    \begin{defn}[Alternating De Bruijn Cycle]\label{def:adbs}
        An \emph{alternating De Bruijn cycle} of order $2n + 1$ on an alphabet pair $(\mathcal{A}, \mathcal{B})$ is a cyclic word that contains every alternating word of length $2n + 1$ on the alphabet pair ($\mathcal{A},\mathcal{B})$ exactly once. It has a length of $2\abs{\mathcal{A}}^{n+1}\abs{\mathcal{B}}^{n}$.
    \end{defn}
    
    \begin{defn}[Unrolled Alternating De Bruijn Cycle]\label{def:audbc}
        If $A_0B_0A_1B_1 \cdots B_{v-1}A_vB_v$ is an alternating De Bruijn cycle, then $A_0B_0A_1B_1 \ldots A_{v-1}B_{v-1}A_vB_vA_0$ is an \emph{unrolled alternating De Bruijn cycle} of the same order for alphabet pair $(\mathcal{A}, \mathcal{B})$.
    \end{defn}

    \begin{rem}
        If the order of an alternating De Bruijn cycle is greater than $3$, then the corresponding unrolled alternating De Bruijn cycle does not cover the same set of words as the alternating De Bruijn cycle from which it was constructed.
    \end{rem}

    We use the following facts about the existence of De Bruijn cycles, universal words, upwords, and upcycles. Upwords and upcycles are considered nontrivial when they contain at least one $\diam$ and at least one letter of $\A$, and universal words and De Bruijn cycles are trivial upwords and upcycles, respectively. 
    \begin{thm}[Van Aardenne-Ehrenfest and De Bruijn \cite{AB51}, Tutte and Smith \cite{TS41}]\label{thm:dbcount}
        For all integers $a\ge 2$ and $n \ge 1$, there are $(a!)^{a^{n-1}}/a^n$ De Bruijn cycles for $\set{0,1,\ldots,a-1}^n$ and therefore $(a!)^{a^{n-1}}$ universal words for $\set{0,1,\ldots,a-1}^n$.
    \end{thm}
    \begin{thm}[Theorems 9, 10, 13, and 17 in Chen, Kitaev,  M\"{u}tze, and Sun \cite{chen2017universal}]\label{thm:upwordexist}
        For all integers $n \ge 2$, there is at least one nontrivial upword for $\set{0,1}^n$.
    \end{thm}
    \begin{thm}[Theorem 5.2 in Goeckner et al. \cite{Goeckner2018}, Fillmore et al. \cite{Fillmore2023}]\label{thm:upcycleexist}
        For all even integers $a \ge 2$ and $n \in \set{4, 8}$, there is at least one nontrivial upcycle for $\set{0,1,\ldots, a-1}^n$ that has diamondicity $1$ and length $a^{n-1}$. Therefore for all even integers $a \ge 2$ and $n \in \set{4, 8}$, there is at least one nontrivial upword for $\set{0,1,\ldots, a-1}^n$.
    \end{thm}
\section{Computationally Identified Upmatrices}\label{sec:results-comp-upmatrices}

    A computational search identified many upmatrices for $\set{0,1}^{2\times 2}$. There are three possible shapes of upmatrices, and the computational search found all three. Upmatrices can be considered a strict superset of uptori, their cyclic counterparts. Any uptorus can be converted to an upmatrix by adding a copy of the leftmost $\ell-1$ columns to the right, and of the topmost $w-1$ rows to the bottom. The resulting upmatrix then shows this repetition (and can be glued together into the uptorus). Some upmatrices, however, have non-identical leftmost and rightmost $\ell-1$ columns, non-identical topmost and bottommost $w-1$ rows, or both; they have no direct conversion to an uptorus. Upmatrices for $\A^{w\times \ell}$ that can be treated cyclically in one direction can be glued together into a cylinder that covers each matrix in $\A^{w\times \ell}$ exactly once.
       
    \begin{example}\label{ex:upmatrices}
    The following three upmatrices for $\set{0,1}^{2\times 2}$ were identified computationally and demonstrate these three shapes: the first upmatrix can be treated cyclically in one direction, the second can be treated cyclically in neither direction, and the third can be treated cyclically in both directions. The vertical and horizontal lines within the upmatrices indicate possible gluings. 
         \begin{center}
        \begin{multicols}{3}
            $\left[\begin{array}{llll|l}
                0 & 0 & 1 & 1& 0\\
                0 & \diam & 1 & \diam & 0 \\
                1 & 0 & 0 & 1 & 1
            \end{array}\right]$
            
            $\begin{bmatrix}
                0 & 0 & 0 & 1 & 1 & 1\\
                0 & 1 & 1 & 0 & 0 & 1\\
                \diam &1&\diam &\diam &0&0\\
            \end{bmatrix}$

            $\left[\begin{array}{llll|l}
                \diam & 0 & 0 & 1 & \diam\\
                1 & 1 & 0 & 0 & 1\\
                1 & 1 & 0 & 0 & 1\\\hline
                \diam & 0 & 0 & 1 & \diam\\
            \end{array}\right]$
        \end{multicols}
    \end{center}

    \end{example}

    As suggested by the above examples, many dimensions are possible for upmatrices for $\set{0,1}^{2\times 2}$. The size (total number of entries) must be between $4$ and $34$. Two sorts of trivial upmatrices for $\set{0,1}^{2\times 2}$ are a $2 \times 2$ matrix containing all $\diam$'s, and matrices containing no $\diam$'s, for example:
    \[
        \begin{bmatrix}
            0 & 0 & 0 & 0 & 0 & 1 & 0 & 1 & 0 & 1 & 0 & 1 & 1 & 1 & 1 & 1 & 0 \\
            0 & 0 & 1 & 1 & 0 & 0 & 0 & 1 & 1 & 0 & 1 & 1 & 0 & 0 & 1 & 1 & 0
        \end{bmatrix}.
    \]

    \begin{example}
        The matrix 
        \[
        \begin{bmatrix}
            0 & 0 & 0 & 0 & 0 & 1 & 0 & 1 & \diam & 1 & \diam \\
            0 & 0 & 1 & 1 & 0 & 0 & 1 & \diam & 0 & 1 & 1
        \end{bmatrix}
        \] is a $2 \times 11$ upmatrix for $\set{0,1}^{2\times 2}$. As the number of rows is the same in the upmatrix and in the matrices it covers, this example shows that \cref{min-dimension} cannot be extended to upmatrices.
    \end{example}

    The full lists of upmatrices for $\set{0,1}^{2\times 2}$ of sizes $3 \times 5$, $3 \times 6$, and $4 \times 4$ are available in \cite{D_Carey_Universal_Partial_Tori}.
    
\section{Upmatrices Constructed from Upwords}\label{sec:results-constructed-upmatrices}

    Given a universal partial word $w$ for $\A^n$ and an integer $p \ge 2$, we construct a $p \times |w|$ matrix $\mu$ whose first row is $w$ and whose subsequent $p-1$ rows contain only $\diam$'s:
        \[
            \mu(w,p) = \begin{bmatrix}
                w\\
                \diam^{|w|}\\
                \vdots\\
                \diam^{|w|}
            \end{bmatrix}.
        \]
    For example, choosing the universal word $0120221100$ for $\set{0,1,2}^2$ and $p=2$ produces the matrix
        \[
            \mu(0120221100,2) = \begin{bmatrix}
                0 & 1 & 2 & 0 & 2 & 2 & 1 & 1 & 0 & 0\\
                \diam & \diam & \diam & \diam & \diam & \diam & \diam & \diam & \diam & \diam\\
            \end{bmatrix},
        \]
    which is a nontrivial upmatrix for $\set{0,1,2}^{2\times 2}$.
    
    \begin{thm}\label{thm:upmatrix-construction}
        If $w$ is an upword for $\A^n$ and $p \ge 2$ is an integer, then $\mu(w,p)$ is a universal partial matrix for $\A^{p \times n}$.
    \end{thm}

    \begin{proof}
        As $\mu(w,p)$ has $p$ rows, its $p \times n$ subarrays are in one-to-one correspondence with the length-$n$ subwords of the first row of $\mu(w,p)$. Any matrix $t \in \A^{p \times n}$ is covered exactly once in $\mu(w,p)$ because its first row is covered exactly once in the first row of $\mu(w,p)$, by the definition of an upword, and the remaining $p-1$ rows of $t$ are covered by $\diam$'s in $\mu(w,p)$.
    \end{proof}

    \begin{cor}
        \cref{thm:upmatrix-construction} constructs infinitely many nontrivial upmatrices. In particular:
        \begin{enumerate}
            \item For all integers $a \ge 2$, $n \ge 1$, and $p\ge 2$, \cref{thm:upmatrix-construction} constructs $(a!)^{a^{n-1}}$ distinct upmatrices for $\set{0,\ldots,a-1}^{p\times n}$ having no diamonds in the top row. 
            \item For all integers $n \ge 2$ and $p \ge 2$, \cref{thm:upmatrix-construction} constructs at least one upmatrix for $\set{0,1}^{p \times n}$ having at least one diamond and at least one letter of the alphabet $\set{0,1}$ in the top row.
            \item For all even integers $a \ge 2$, $n \in \set{4,8}$, and all integers $p \ge 2$, \cref{thm:upmatrix-construction} constructs at least one upmatrix for $\set{0,\ldots,a-1}^{p \times n}$ having at least one diamond and at least one letter of the alphabet $\set{0,\ldots,a-1}$ in the top row.
        \end{enumerate}
    \end{cor}

    \begin{proof}
        By Theorems \ref{thm:dbcount}, \ref{thm:upwordexist}, and \ref{thm:upcycleexist}, respectively. 
    \end{proof}

\section{Computationally Identified Uptori}\label{sec:results-comp-uptori}
    A computational search identified 
    \[
        \begin{bmatrix}
            \diam & 0 & 0 & 1\\
            1 & 1 & 0 & 0\\
            1 & 1 & 0 & 0\\
        \end{bmatrix}
    \]
    as a minimal example of a nontrivial uptorus for $\set{0,1}^{2\times 2}$ in its dimensions and number of diamonds.

    First, all nontrivial uptori (as defined in \cref{sec:prelim}) have at least one $\diam$ and $w, \ell \ge 2$. The $3 \times 4$ dimensions of the uptorus are also the minimum. This can be seen as, by Lemma \ref{min-dimension}, the minimum dimensions of a candidate torus are $3\times 3$, so we need only consider that case here. There are sixteen matrices in $\set{0,1}^{2 \times 2}$. An uptorus must cover each exactly once. With no $\diam$'s, such a torus would cover 9 matrices in $\set{0,1}^{2\times 2}$. With one $\diam$ it would cover 13 matrices in $\set{0,1}^{2\times 2}$, too few. With two $\diam$'s it would cover 18 (or 19, if the two $\diam$'s are adjacent), too many.

    \begin{rem}
        The minimal example we describe here represents many equivalent uptori. We consider two uptori to be part of the same equivalence class if they are related by horizontal reflection, vertical reflection, the permutation of the alphabet (excluding $\diam$'s), column-wise rotation, row-wise rotation, transposition, or compositions of those operations. For example, the following uptori are equivalent to each other and to the above minimal example:
        \begin{multicols}{4}
            $\begin{bmatrix}
                \diam & 1 & 1 & 0\\
                0 & 0 & 1 & 1\\
                0 & 0 & 1 & 1\\
            \end{bmatrix}$
        
            $\begin{bmatrix}
                1 & 0 & 0 & \diam\\
                0 & 0 & 1 & 1\\
                0 & 0 & 1 & 1\\
            \end{bmatrix}$

            $\begin{bmatrix}
                1 & 1 & 0 & 0\\
                1 & 1 & 0 & 0\\
                \diam & 0 & 0 & 1\\
            \end{bmatrix}$

            $\begin{bmatrix}
                \diam & 1 & 1 \\
                0 & 1 & 1 \\
                0 & 0 & 0 \\
                1 & 0 & 0 \\
            \end{bmatrix}$
    \end{multicols}
    \end{rem}

    A computational search identified this uptorus for the $2\times2$ subarray size. We have not ruled out the existence of uptori for the $2\times3$ and $3\times3$ window sizes, but have also not identified them computationally. In the next section we show how to construct uptori for larger matrices, which would be infeasible by an exhaustive computational search.
    
\section{Uptori Constructed from Upcycles}\label{sec:results-constructed-uptori}

    In this section, we offer a construction of uptori similar to the constructions of De Bruijn tori developed in \cite{Cock1988} and \cite{Kreitzer2024}. We define a torus whose rows are particular relative rotations of a given cyclic partial word and then show that for a suitable cyclic partial word and suitable sequence of relative rotations, the resulting torus is a universal partial torus.

    \begin{defn}\label{def:matrix-m(u,s)}
        For a given cyclic partial word $u$ and a word $s$ over the alphabet $\set{0, 1, \ldots, \abs{u}-1}$, the matrix $m(u,s)$ is a rectangular matrix whose rows are defined by a recursive function $\phi$:

        \[
            \phi_n = \begin{cases}
                \sigma^{s_0}(u) & n = 0\\
                \sigma^{s_{n}}(\phi_{n-1}) & n \in \{1 , \ldots , \abs{s}-1\}.
            \end{cases}
        \]
        Matrix $m(u,s)$ is then given by the concatenation of those rows:
        \[
            m(u,s) = \left[\begin{matrix}
                \phi_0\\
                \phi_1\\
                \vdots\\
                \phi_{\abs{s}-1}
            \end{matrix}\right]
        \]
    \end{defn}

    Note that $m$ has exactly $\abs{s}$ rows and exactly $\abs{u}$ columns, and the same alphabet as $u$. For a given cyclic partial word $u$, we can construct infinitely many different matrices $m$.

    \begin{example}[Matrix $m(u,s)$]\label{ex:upcycle-uptorus}
        Let $u = 001\diam110\diam$ and
        \[
            s = 0017020304050607112722313733414247445152535755616263646776654321.
        \]
        We construct $m(u,s)$ and represent $m(u,s)^{\intercal}$ as an image, with the black pixels representing $0$'s, the light grey pixels $1$'s, and the red pixels $\diam$'s. (We return to this example in \cref{rem:smallestupcycle} and \cref{ex:smalluptorus} and use the fact that $u$ is an upcycle for $\set{0,1}^4$ and $s$ is a De Bruijn cycle for $\set{0, \ldots, 7}^2$, along with \cref{thm:CoversAllWxLMatricesExactlyOnce}, to conclude that $m(u,s)$ is an uptorus for $3\times 4$ matrices.)

        \begin{figure}[!h]
            \centering
            \includegraphics[width=6.5in]{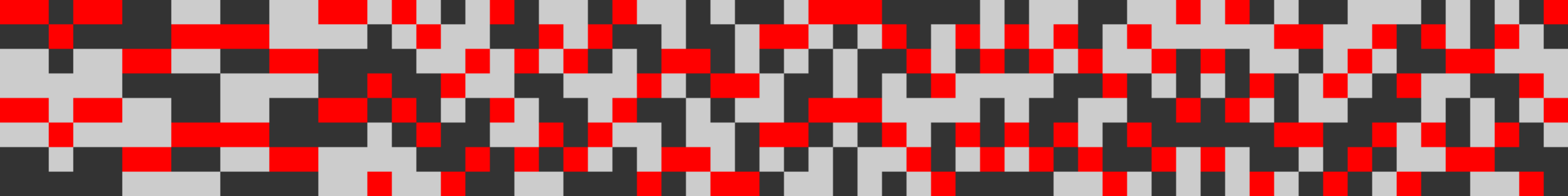}   
            \caption{The transpose of matrix $m(u,s)$}
        \label{fig:uptorus-4x3}
        \end{figure}
    \end{example}

    We use the following lemma, similar to \cite[Lemmas 2 and 3]{Kreitzer2024} and \cite{Cock1988}.
    \begin{lem}\label{lem:s0}
        For every cyclic partial word $u$ and De Bruijn cycle $s$ for $\set{0, 1, \ldots, \abs{u}-1}^y$ with $y \ge 2$, the bottom row of $m(u,s)$ is $u$. Thus, the relative rotations of cyclically consecutive rows of $m(u,s)$ are given by $s$.
    \end{lem}

    \begin{proof}
        The recursive definition of $m(u,s)$ guarantees that the linearly consecutive rows of $m(u,s)$ have relative rotations given by $s_1\cdots s_n$. It remains to show that the relative rotation from the bottom row of $m(u,s)$ to the top row is given by $s_0$, i.e., that the bottom row is $u$.

        Let $\xi = s_0 + s_1 + \ldots + s_{\abs{s}-1}$. The definition of $m(u,s)$ implies that a row with index $j$, where $0 \le j \le \abs{s}-1$, is rotated a total of $s_0 + \ldots +      s_j = \sum_{k=0}^j s_k$ units (mod $\abs{u}$), relative to $u$. The bottom row has index $\abs{s}-1$ and therefore is rotated $\xi$ units relative to $u$. Therefore, we only need to show that $\xi \equiv 0 \mod \abs{u}$.
       
        Since $s$ is a De Bruijn cycle, each number in $\set{0, 1, \ldots, \abs{u}-1}$ occurs the same number of times, which is the length of $s$ divided by the alphabet size of $s$, i.e., $\abs{u}^y/\abs{u} = \abs{u}^{y-1}$. Then 
            
        \[
            \xi = \sum_{k=0}^{\abs{u}-1} \abs{u}^{y-1}k = \abs{u}^{y-1}\sum_{k=0}^{\abs{u}-1}k = \abs{u}^{y-1}\binom{\abs{u}}{2}.
        \]
            
        As $\binom{\abs{u}}{2}$ is an integer and $y \ge 2$, we have that $\xi$ is an integer multiple of $\abs{u}$, and $\xi \equiv 0 \pmod{\abs{u}}$.
    \end{proof}

    Now we are ready to prove that $m(u,s)$, considered cyclically in both dimensions, is an uptorus for suitable inputs.
    
    \begin{thm}\label{thm:CoversAllWxLMatricesExactlyOnce}
        Given an upcycle $u$ for $\A^x$ and a De Bruijn cycle $s$ for $\set{0, \ldots, \abs{u}-1}^y$, $y \geq 2$, the torus $m(u,s)$ covers every distinct subarray in $\A^{(y+1) \times x}$ exactly once.
    \end{thm}

    \begin{proof}
        For an arbitrary matrix $P$ in $\A^{(y+1) \times x}$, we define $f(P)$ as follows. Note that $P$ has as many columns as the length of words that $u$ covers (i.e. $x$ columns) and that the alphabet of $P$ is the alphabet of $u$. Let $P_0, \ldots, P_y$ be the rows of $P$. Because each row of $P$ is a word in $\A^x$, each row of $P$ is covered by $u$ in exactly one place.
        
        For each row of $P$, say, $P_n$, there is exactly one integer $a_n \in \set{0 , \ldots , |u|-1}$ such that $\sigma^{a_n}(u)$ covers $P_n$ in its first $x$ characters. We use the sequence $(a_n)_{n=0}^{y}$ to construct a sequence of differences, $b = (b_n)_{n=0}^{y-1}$, where $b_n = a_{n+1} - a_{n}$ for $n \in \set{0 , \ldots , y-1}$. Note that $b \in \set{0,\ldots,\abs{u}-1}^y$, and therefore $b$ appears in $s$ in exactly one place, say $b = s_{i}\cdots s_{i+y-1}$, where $i \in \set{0, \ldots, \abs{s}-1}$. Then $f(P):=\left(i-1,a_0 - \sum_{k=0}^{i-1}s_k\right)$.
        
        Now we show that $P$ is covered in the $(y+1)\times x$ subarray of $m(u,s)$ whose top-left corner has coordinates $f(P)$. By the definition of $m(u,s)$ and \cref{lem:s0}, for every integer $q'$, the ${q'}^\text{th}$ row of $m(u,s)$ is the $q:=(q' \mod \abs{s})$th row of $m(u,s)$, which is $\phi_{q} = \sigma^{\sum_{k=0}^{q}s_k} (u)$, and we let $\phi_{q'}:=\phi_q$. Recall also that we defined $i$ in such a way that for every integer $0 \le n \le y-1$, we have $s_i\cdots s_{i-1+n} = b_0 \cdots b_{n-1}$.
        \begin{align*}
            \sigma^{a_0 - \sum_{k=0}^{i-1}s_k}(\phi_{i-1+n}) &= \sigma^{a_0 - \sum_{k=0}^{i-1}s_k}(\sigma^{\sum_{k=0}^{i-1+n}s_k}(u))\\
            &= \sigma^{a_0+\sum_{k=i}^{i-1+n}s_k}(u)\\
            &= \sigma^{a_0+\sum_{k=0}^{n-1}b_k}(u)\\
            &= \sigma^{a_n}(u)\quad\text{by the definition of }b_n,
        \end{align*}
        and $P_n$ is covered at the beginning of $\sigma^{a_n}(u)$ by the definition of $a_n$. The column index at the beginning of where $P_n$ is covered within the $i-1+n$th row of $m(u,s)$ is the amount by which we shift $\phi_{i-1+n}$ to get $P_n$ at the beginning, so for every $n$ the row $P_n$ is covered in $m(u,s)$ starting at index $a_0 - \sum_{k=0}^{i-1}s_k$.

        In any other set of $y+1$ consecutive rows of $m(u,s)$, the relative rotations of $u$ in those rows are different (because they are given by a De Bruijn cycle, $s$), so these rows do not cover $P$. Hence, $P$ is covered exactly once in $m(u,s)$.
    \end{proof}

    \begin{rem}\label{rem:smallestupcycle}
        Every upcycle has subword length $x \ge 4$ (proven in \cite{Goeckner2018}, Corollary 4.12 and Proposition 5.1), and $y+1 \ge 3$ is a hypothesis of \cref{thm:CoversAllWxLMatricesExactlyOnce}, so the smallest uptori that can be constructed by \cref{thm:CoversAllWxLMatricesExactlyOnce} are for $3 \times 4$ matrices, like the one shown in \cref{ex:upcycle-uptorus}.
        
        It is not known whether any upcycles exist for odd alphabet sizes. If they do, then a modified form of \cref{lem:s0} may allow for the construction of uptori with $2\times x$ windows, using upcycles over odd alphabets.
    \end{rem}

    \begin{rem}
        If an upcycle exists for a given alphabet and window length, then this construction produces uptori for subarrays whose width is that window length and whose height is greater than two.
    \end{rem}
    
    \begin{example}\label{ex:smalluptorus}
        Consider matrix $m(u,s)$ where $u = 001\diam110\diam$, an upcycle for $\set{0,1}^4$, and $s$ is a De Bruijn cycle for $\set{0, \ldots, 7}^2$, as in \cref{ex:upcycle-uptorus}. Let $P$ be the $3 \times 4$ matrix:
        
            \[
                \left[\begin{matrix}
                    0 & 0 & 1 & 1\\
                    1 & 0 & 1 & 0\\
                    1 & 0 & 0 & 1
                \end{matrix}\right]
            \]

        Each row of $P$ is covered in the first four characters of some rotation of $u$:
        \begin{multicols}{3}
            $\sigma^{0}(u) = \textcolor{red}{001\diam}110\diam$
        
            $\sigma^{5}(u) = \textcolor{red}{10\diam0}01\diam1$
        
            $\sigma^{7}(u) = \textcolor{red}{\diam001}\diam110$
        \end{multicols}
            
        We can then locate $P$ by finding $a_0 = 0$, $a_1 = 5$, and $a_2 = 7$. The sequence of differences $b$ is $b_0 = a_1 - a_0 = 5$ and $b_1 = a_2 - a_1 = 2$. Examining $m(u,s)$, we find two pairs of consecutive rows whose relative rotations are 5 and 2 respectively:

        \[
            \begin{matrix}
            \phi_{0}&0&0&1&\diam&1&1&0&\diam \\

            \vdots\\
            
            \phi_{i-3}& 1& 1& 0& \diam& 0& 0& 1& \diam \\

            \phi_{i-2}& 0& 1& \diam& 1& 1& 0& \diam& 0 \\

            \phi_{i-1}& \p{1}& \p{\diam}& 1& 1& 0& \diam& \p{0}& \p{0} \\

            \phi_{i}& \p{\diam}& \p{0}& 0& 1& \diam& 1& \p{1}& \p{0} \\

            \phi_{i+1}& \p{0}& \p{1}& \diam& 1& 1& 0& \p{\diam}& \p{0} \\

            \phi_{i+2}& 0& \diam& 0& 0& 1& \diam& 1& 1 \\

            \phi_{i+3}& 0 &1 &\diam &1 &1 &0 &\diam &0\\
            \vdots
            \end{matrix}
        \]

        The location of $P$ in $m(u,s)$ is highlighted in red.

    \end{example}

    \begin{cor}\label{cor:infinitude}
        \cref{thm:CoversAllWxLMatricesExactlyOnce} constructs infinitely many nontrivial uptori. In particular, for all even integers $a\ge 2$, all integers $y \ge 2$, and $x \in \set{4,8}$, \cref{thm:CoversAllWxLMatricesExactlyOnce} constructs at least $(((a^{x-1})!)^{a^{(x-1)(y-1)}})/a^{(x-1)y}$ nontrivial uptori for $\set{0,\ldots,a-1}^{(y+1)\times x}$.
    \end{cor}
    \begin{proof}
        By \cref{thm:upcycleexist}, there is at least one nontrivial upcycle $u$ for $\set{0,\ldots,a-1}^x$ of length $\abs{u}=a^{x-1}$. Then there are $(((a^{x-1})!)^{a^{(x-1)(y-1)}})/a^{(x-1)y}$ De Bruijn cycles $s$ for $\set{0, \ldots, \abs{u}-1}^y$ by \cref{thm:dbcount}. For each such pair $(u,s)$, \cref{thm:CoversAllWxLMatricesExactlyOnce} constructs a nontrivial uptorus $m(u,s)$ for $\set{0,\ldots,a-1}^{(y+1)\times x}$.
        
        Given an uptorus $m$ constructed in this way, $a$ can be determined by the number of distinct characters in $m$ (each letter $j \in \set{0, \ldots, a-1}$ must appear in $u$, and therefore in $m$, because otherwise $j^x$ would be covered in $u$ by $\diam^x$, which covers all words of length $x$, and the upcycle $u$ would be trivial). Then $x$ can be determined because the number of columns of $m$ is $\abs{u} = a^{x-1}$. Finally $y$ can be determined because the number of rows of $m$ is $\abs{u}^y = a^{(x-1)y}$. Therefore distinct values of $a$, $x$, and $y$ produce distinct uptori, and the total number of nontrivial uptori constructed by \cref{thm:CoversAllWxLMatricesExactlyOnce} is infinite.
        
        For a given triple $(a,x,y)$ and a given upcycle $u$ for $\set{0,\ldots,a-1}^x$, distinct De Bruijn cycles $s$ for $\set{0, \ldots, a^{x-1}-1}^y$ also produce distinct uptori by the definition of $m(u,s)$ and \cref{lem:s0}: the relative rotations of cyclically consecutive rows of $m(u,s)$ are given by $s$.
    \end{proof}

    \begin{thm}\label{lem:DefinedDiamondicity}
        Given an upcycle $u$ for $\A^x$ and a De Bruijn cycle $s$ for $\set{0, \ldots, \abs{u}-1}^y$ with $y \ge 2$, the torus $m(u,s)$ has well-defined diamondicity. That is, every $(y+1) \times x$ subarray of $m$ has the same number of wildcard characters.
    \end{thm}

    \begin{proof}
        Let $P$ be an arbitrary $(y+1)\times x$ subarray of $m(u,s)$. Each row of $P$ is a partial word contained as a substring in $u$. As noted in \cref{sec:prelim}, every upcycle has well-defined diamondicity. Let $d$ be the diamondicity of $u$. Then each row of $P$ contains $d$ wildcard characters, and $P$ itself contains $d(y+1)$ wildcard characters. As $P$ was arbitrary, $m(u,s)$ has diamondicity $d(y+1)$.
    \end{proof}

\section{Uptori Constructed from Universal Partial Families}\label{sec:upfamilies}

    In this section we offer another construction of uptori similar to the construction of De Bruijn tori developed by Kreitzer et al. \cite{Kreitzer2024}. They use De Bruijn families: sets of cyclic words that, when taken as a group, have the covering property of a De Bruijn cycle. We extend that notion to sets of cyclic words that include $\diam$'s: universal partial families.

    We then define a matrix whose rows are particular rotations of particular elements of an universal partial family, and then show that the resulting matrix, read cyclically, is a universal partial torus.

    \begin{defn}[Universal Partial Family]\label{def:upfamily}
        A \emph{universal partial family} (or \emph{upfamily}), for $\A^{\ell}$ is a set $\mathcal{S}$ of cyclic partial words over the alphabet $\A$ such that every element of $\mathcal{S}$ has the same length, and every word in $\A^{\ell}$ is covered by exactly one element of $\mathcal{S}$, and is covered exactly once in that element.
    \end{defn}

        If $\abs{\A} = 1$, then the only upfamilies are $\set{0}$ and $\set{\diam}$, which we call trivial and exclude from our consideration here.

    \begin{example}[Universal Partial Family for $\set{0,1,2,3}^4$]\label{ex:upfamily}
        The set
        \begin{align*}
            \mathcal{F} = \set{\,
             &001\diam110\diam, \quad
             003\diam112\diam, \quad
             021\diam130\diam, \quad
             023\diam132\diam,\\
             &201\diam310\diam,\quad
             203\diam312\diam,\quad
             221\diam330\diam,\quad
             223\diam332\diam\,
            }
        \end{align*}
        covers every word in $\set{0,1,2,3}^4$ exactly once.
    \end{example}

    The upfamily in \cref{ex:upfamily} is derived from a related upcycle, first constructed in \cite{Goeckner2018}, and in \cite{Fillmore2023} shown to be the result of using the alphabet multiplier theorem with the upcycle $001\diam110\diam$ and multiplier $k=2$:
    \begin{equation}\label{eq-upcycle}
        u = 001\diam110\diam003\diam112\diam021\diam130\diam023\diam132\diam201\diam310\diam203\diam312\diam221\diam330\diam223\diam332\diam.
    \end{equation}

    \begin{rem}
    Note that the elements of $\F$ appear consecutively in $u$, and that, by inspection, each word in $\set{0,1,2,3}^4$ appears exactly once in exactly one element of $\F$. Particular words are not necessarily covered by the same parts of $\F$ and $u$. For example, $2000$ is covered in element $\textcolor{red}{00}3\diam11\textcolor{red}{2\diam}$ in $\F$, but the corresponding substring of $u$, $\ldots1\textcolor{red}{0\diam00}3\diam11\textcolor{red}{2\diam02}1\ldots$, does not cover $2000$.
    \end{rem}

    \begin{conj}\label{conj:upfamily-from-alphabet-multiplier}
        If $v$ is an upcycle constructed by the alphabet multiplier theorem, then there exists at least one upfamily $\F$ constructed by slicing $v$ into identically sized elements.
    \end{conj}

    We have computationally identified such upfamilies for the upcycle $001\diam110\diam$ and multipliers $k = 2$, $3$, $4$, and $5$. In fact, for these examples, the following stronger conjecture holds.

    \begin{conj}\label{conj:upfamily-pizza-slice}
        Let $v$ be an upcycle constructed by the alphabet multiplier theorem from upcycle $u$ and multiplier $k$. For every $0 \le i \le n-d$, if $v$ is sliced into cyclic partial words of length $k^i\abs{u}$ starting at any index of $v$, then the result is an upfamily.
    \end{conj}

    For example, for the particular upcycle $u$ from \cref{eq-upcycle}, varying the starting index of the slicing produces different upfamilies:
    \begin{multicols}{3}
        \begin{align*}
            \mathcal{F} = \{\,
            & 001\diam110\diam, \\
            & 003\diam112\diam, \\
            & 021\diam130\diam, \\
            & 023\diam132\diam,\\
            & 201\diam310\diam,\\
            & 203\diam312\diam,\\
            & 221\diam330\diam,\\
            & 223\diam332\diam\,\}
        \end{align*}

        \begin{align*}
            \mathcal{F'} = \{\,
            & 01\diam110\diam0, \\
            & 03\diam112\diam0, \\
            & 21\diam130\diam0, \\
            & 23\diam132\diam2,\\
            & 01\diam310\diam2,\\
            & 03\diam312\diam2,\\
            & 21\diam330\diam2,\\
            & 23\diam332\diam0\,\}
        \end{align*}

        \begin{align*}
            \mathcal{F''} = \{\,
            & 1\diam110\diam00, \\
            & 3\diam112\diam02, \\
            & 1\diam130\diam02, \\
            & 3\diam132\diam20,\\
            & 1\diam310\diam20,\\
            & 3\diam312\diam22,\\
            & 1\diam330\diam22,\\
            & 3\diam332\diam00\,\}
        \end{align*}
    \end{multicols}

    Additionally, $u$ can be sliced into an upfamily with four members:
    \begin{align*}
        \mathcal{F} = \{\,
        & 001\diam110\diam003\diam112\diam, \\
        & 021\diam130\diam023\diam132\diam, \\
        & 201\diam310\diam203\diam312\diam,\\
        & 221\diam330\diam223\diam332\diam\,\}
    \end{align*}
    
    Having identified several upfamilies we now turn to their use in a construction of uptori.

    \begin{defn}\label{def:m-W}
        Let $\F = \set{F_0, \ldots, F_{\abs{\F}-1}}$ be an upfamily for $\A^x$ and $s$ the set $\set{0, 1, \ldots , \abs{F_0} - 1}$. Let $\W$ be an alternating word for alphabet pair $(\F, s)$ whose length is $2v + 1$. For clarity's sake, we index $\W$ like so:
        \[
            \W = f_0 r_0 f_1 r_1 \cdots r_{v-1} f_{v}
        \]
        with $f_i \in \F$ for all $i \in \set{0, \ldots, v}$ and $r_j \in s$ for all $j \in \set{0, \ldots, v-1}$. The function $\phi$ is defined as follows:

            \[
                \phi_n = 
                \begin{cases}
                    f_0 & n = 0\\
                    \sigma^{\sum_{i=0}^{n-1}{r_i}}(f_{n}) & n \in \set{1 , \ldots , v}
                \end{cases}
            \]

        Matrices $m'(\W)$ and $m(\W)$ are then given by the concatenation of those rows, with the latter having one fewer row:
            \[
                m'(\W) = 
                    \begin{bmatrix}
                        \phi_0\\
                        \phi_1\\
                        \vdots\\
                        \phi_{v}
                    \end{bmatrix}\quad\text{and}\quad
                m(\W) = 
                    \begin{bmatrix}
                        \phi_0\\
                        \phi_1\\
                        \vdots\\
                        \phi_{v-1}
                    \end{bmatrix}.
            \]
    \end{defn}        

    We show that if $\F$ is an upfamily for $\A^x$, $s$ is the set $\set{0, 1, \ldots , \abs{F_0} - 1}$, and alternating word $\W$ for alphabet pair $(\F,s)$ is an unrolled alternating De Bruijn cycle of order $2y+1$, with $y \geq 2$, then the first row of matrix $m'(\W)$ is equal to the last row, and matrix $m(\W)$ can be treated as the torus with $v$ rows and $\abs{F_0}$ columns resulting from gluing together these first and last rows of $m'(\W)$. We show that $m(\W)$ is furthermore an uptorus for $\A^{(y+1)\times x}$ (though note that $m'(W)$ is not an upmatrix; it would be necessary to unroll $m(\W)$ further in both horizontal and vertical directions to obtain an upmatrix). First, we establish the conditions for $m(\W)$ to be a cyclic form of $m'(\W)$ in the following lemma, similar to  \cref{lem:s0} and \cite[Lemmas 2 and 3]{Kreitzer2024}.

    \begin{lem}\label{lem:m(W)-is-torus}
        Let $\W$ be an unrolled alternating De Bruijn cycle of order $2y+1$ on the alphabet pair $(\F,s)$. If either
        \begin{enumerate}[(a)]
            \item $y\geq2$, or
            \item the number of elements of the upfamily, $\abs{\F}$, is even,
        \end{enumerate}
        then $\phi_v = f_0 = \phi_0$.
    \end{lem}

    \begin{proof}
        Our construction of the unrolled alternating De Bruijn cycle guarantees that $f_0$ and $\phi_v$ are each some rotation of the same element of $\F$, so it suffices to show that the rotation applied to $f_0$ to obtain $\phi_v$ is equivalent to the identity function. Thus it suffices to show that $\xi:=\sum_{i=0}^{v-1}r_i$ is equivalent to $0 \pmod{\abs{F_0}}$, whenever either of the two conditions is met.

        The number of rotation indices is $v = \abs{\F}^{y+1} \cdot \abs{F_0}^y$ because this is the number of alternating words of length $2y+1$ on the alphabet pair $(\F,s)$.

        We show that each element of $s$ occurs equally often in $\W$. Index $\W$ as $\W = f_0 r_0 f_1 r_1 \cdots r_{v-1} f_{v}$. Let $\W_C$ be the alternating De Bruijn cycle $\W = f_0 r_0 f_1 r_1 \cdots r_{v-1}$, where the subscripts are considered mod $2v$ (because $\W_C$ is a cyclic object). Each location in $\W_C$ where an element $s_0\in s$ occurs corresponds to an alternating word whose second element (i.e. whose first element from the alphabet $s$) is in the location that is given. There are $\abs{\F}^{y+1} \cdot \abs{F_0}^{y-1}$ words of length $2y+1$ whose second element is $s_0$; each of those words corresponds to a unique location (up to equivalence mod $2v$) because $\W_C$ is an unrolled alternating De Bruijn cycle. Therefore, there are $\abs{\F}^{y+1} \cdot \abs{F_0}^{y-1}$ occurrences of $s_0$ in $\W_C$. As $\W$ and $\W_C$ differ only by an element of $\F$, they contain the same number of copies of $s_0$. Since $s_0$ was arbitrary, each element of $s$ occurs equally often.

        Therefore, the average rotation is $\frac{1}{2}\left(\abs{F_0}-1\right)$. The sum of the rotation indices is therefore 
        \begin{align*}
            \xi &= r_0 + r_1 + \cdots + r_{v-1} \\
            &= v\cdot \frac{1}{2}\left(\abs{F_0}-1\right) \\
            &= \abs{\F}^{y+1}\cdot \abs{F_0}^y \cdot \frac{1}{2}\left(\abs{F_0}-1\right)\\
            &=\abs{F_0} \cdot \left( \abs{\F}^{y+1}\cdot \abs{F_0}^{y-1} \cdot  \frac{1}{2}\left(\abs{F_0}-1\right) \right)
        \end{align*}

        It is thus sufficient to show that $\abs{\F}^{y+1} \cdot \abs{F_0}^{y-1} \cdot  \frac{1}{2}\left(\abs{F_0}-1\right)$ is an integer, so that $\xi$ is a multiple of $\abs{F_0}$. 

        Condition (a) states that $y \geq 2$, so $\abs{F_0}^{y-2}$ is an integer, and therefore
        \[
            \abs{\F}^{y+1} \cdot \abs{F_0}^{y-1} \cdot  \frac{1}{2}\left(\abs{F_0}-1\right) = \abs{\F}^{y+1} \abs{F_0}^{y-2}\cdot \abs{F_0}\cdot \frac{1}{2}\cdot \left(\abs{F_0}-1\right) =  \abs{\F}^{y+1} \abs{F_0}^{y-2}\cdot\binom{\abs{F_0}}{2}
        \]
        is an integer. Therefore, the claim is true when condition (a) is met. 

        Condition (b) states that $\abs{\F}$ is even. Rewriting our expression for the sum of the rotation indices gives
        \[
            \xi =\left( \abs{\F}^{y+1}\cdot \frac{1}{2}\right)\cdot \abs{F_0}^{y-1}\cdot \left(\abs{F_0}-1\right)\cdot \abs{F_0}.
        \]

        Because $y+1\geq1$ and $\abs{\F}$ is even, $\left( \abs{\F}^{y+1} \cdot \frac{1}{2}\right)$ is an integer. Additionally, $y-1 \geq0$,  $\abs{F_0}^{y-1} \cdot \left(\abs{F_0}-1\right)$ is an integer. Therefore, $\xi$ is divisible by $\abs{F_0}$, and the claim is true when condition (b) is met.
    \end{proof}

    Now we can use upfamilies and alternating De Bruijn cycles to construct uptori.
    \begin{thm}\label{thm:m(W)-covers-exactly-once}
        If $\F$ is an upfamily for $\A^x$, $s = \set{0, 1, \ldots , x - 1}$, and $\W$ is an unrolled alternating De Bruijn cycle for alphabet pair $(\F,s)$ of order $2y+1$, then $m(\W)$ is an uptorus for $\A^{(y+1) \times x}$.
    \end{thm}

    \begin{proof}

    Index $\W$ as  $\W = f_0 r_0 f_1 r_1 \cdots r_{v-1} f_{v}$. Let $\W_C$ be the alternating De Bruijn cycle $ \W = f_0 r_0 f_1 r_1 \cdots r_{v-1}$, where the subscripts are considered mod $2v$.

        Let $P \in \A^{(y+1)\times x}$. Each row of $P$ is a word in $\A^x$ and therefore covered exactly once in exactly one member of $\F$ (by the definition of an upfamily). $P$ is then uniquely identified by: a sequence of $y+1$ members of $\F$, a relative rotation that causes the top row of $P$ to be covered in the first $x$ elements of the corresponding cycle of $\F$ (the first cycle in the sequence, for the top row), and a sequence of $y$ relative rotations that cause each lower row of $P$ to appear in the first $x$ characters of the corresponding element of $\F$.

        Consider an arbitrary row of $P$, say $P_n$. Exactly one member of $\F$  covers $P_n$, say $F_{c_n}$. The sequence $c=(c_n)_{n=0}^{y}$ describes which elements of $\F$ cover the rows of $P$ and has exactly $y+1$ members.

        Further, there is exactly one integer $a_n \in \set{0 , \ldots , \abs{F_{c_n}}-1}$ such that $\sigma^{a_n}(F_{c_n})$ covers $P_n$ in its first $x$ characters. (Recall that by the definition of an upfamily, $\abs{F_0} = \abs{F_{c_n}}$.) We use the sequence $(a_n)_{n=0}^{y}$ to construct a sequence of differences, $b = (b_n)_{n=0}^{y-1}$, where $b_n = a_{n+1} - a_{n}$ for $n \in \set{0 , \ldots , y-1}$. Note that sequence $b$ has exactly $y$ members.

        Construct alternating word $w$ from sequence $c$ and sequence $b$: $w = F_{c_{0}}b_{0}F_{c_{1}}b_{1} \cdots b_{y-1}F_{c_{y}}$. Note that $w$ is an alternating word for alphabet pair $(\F,s)$ and has length $2y+1$ and therefore appears exactly once in $\W_C$ (up to equivalence mod $2v$) at, say, $f_i r_i \cdots r_{i+y-1} f_{i+y}$. By the construction of $m(\W)$, $P$ then is covered in $m(\W)$ in rows $\phi_i, \ldots, \phi_{i+y}$ (taking the indices mod $v$, which is possible due to \cref{lem:m(W)-is-torus}).

        Having uniquely identified the rows of $m(\W)$ in which $P$ is covered, we uniquely identify the columns of $m(\W)$ in which $P$ is covered.
        
        For each $n \in \set{0,\ldots,y}$, row $P_n$ is covered in $\phi_{i+n}$, which is $\sigma^{\sum_{j=0}^{i+n-1}r_j}(f_{i+n})$. 
         By the definition of $m(\W)$ and \cref{lem:m(W)-is-torus}, for every integer $q'$, the ${q'}^\text{th}$ row of $m(\W)$ is the $q:=(q'\mod v)^\text{th}$ row of $m(\W)$, so is $\phi_{q} = \sigma^{\sum_{i=0}^{q-1} r_i}(f_q)$, and we let $\phi_{q'} := \phi_q$.
        Independently of $n$, the difference between the rotation required to bring $P_0$ to the first character of $F_{c_0}$ and the rotation of $\phi_{i}$ then gives us the horizontal coordinate where the leftmost character of $P_n$ is covered: 
        \begin{align*}
            \sigma^{a_0-\sum_{j=0}^{i-1}r_j}(\phi_{i+n}) &= \sigma^{a_0-\sum_{j=0}^{i-1}r_j}(\sigma^{\sum_{j=0}^{i+n-1}r_j}(f_{i+n})) = \sigma^{a_0 + \sum_{j=i}^{i+n-1}r_j}(f_{i+n})\\
            &= \sigma^{a_0 + \sum_{j=0}^{n-1}b_j}(f_{i+n})= \sigma^{a_n}(f_{i+n}) = \sigma^{a_n}(F_{c_{i+n}}),
        \end{align*}
        so $P_n$ appears in $m(\W)$ starting at $(i+n,  
        a_0 - \sum_{j=0}^{i-1}r_j)$. Therefore $P$ is covered in exactly one place in $m(\W)$, with its top-left corner covered at $(i, a_0 - \sum_{j=0}^{i-1}r_j)$.
    \end{proof}

    For many values of $\A$, $w$, and $\ell$, both this construction and the construction in \cref{sec:results-constructed-uptori} produce an uptorus for $\A^{w \times \ell}$. For given $\A$, $w$, and $\ell$, the two constructions do not necessarily produce uptori of the same width and height. Additionally, for upfamilies with an even number of members, the upfamily-based construction produces uptori for $\A^{2 \times \ell}$. \cref{thm:CoversAllWxLMatricesExactlyOnce} does not guarantee the existence of uptori with $2\times \ell$ windows by itself (see \cref{rem:smallestupcycle}).

\section{Universal Partial Quasi-Families}\label{sec:quasifamilies}

    Non-perfect factors of the De Bruijn graph (equivalently, a relaxation of De Bruijn families where the cycles may have different lengths) have been studied \cite{ABMPS23,Etzion1988}.
    
    \begin{defn}[De Bruijn Quasi-Family]\label{def:dbquasifamily}
    A \emph{De Bruijn quasi-family} (also called a factor of the De Bruijn graph \cite{ABMPS23,Etzion1988}) for $\A^{\ell}$ is a set $\mathcal{S}$ of cyclic words over $\A$ such that each word in $\A^{\ell}$ is covered by exactly one element of $\mathcal{S}$, and is covered exactly once in that element. 
    \end{defn}

    The elements of a De Bruijn quasi-family are not necessarily all the same length. 
    While the cycle lengths in a De Bruijn family must exceed the word length $n$ (as otherwise $0^n$ would either never be covered or be covered more than once), this is not true of quasi-families. For example

    \begin{center}
        \begin{multicols}{2}
            $\set{00001, 01011, 001111}$
        
            $\set{001, 000011, 0101111}$
        \end{multicols}
    \end{center}
    are both De Bruijn quasi-families for $\set{0,1}^4$. 
        In the second example above, the cyclic word $001$ covers the words $0010$, $0100$, and $1001$.

    We generalize this concept and consider sets of cyclic partial words.
    
    \begin{defn}[Universal Partial Quasi-Family]\label{def:upquasifamily}
        A \emph{universal partial quasi-family} (or \emph{UPQF}) for $\A^{\ell}$ is a set $\mathcal{S}$ of cyclic partial words over $\A$ such that every word in $\A^{\ell}$ is covered by exactly one element of $\mathcal{S}$, and is covered exactly once in that element.
    \end{defn}
    
    In contrast to a universal partial family, the elements of a universal partial quasi-family are not necessarily all the same length.
    
    Like upfamilies, universal partial quasi-families can sometimes be constructed by slicing an upcycle. For example, take the same upcycle, given in \cref{sec:upfamilies}, \cref{eq-upcycle}, that was used to produce \cref{ex:upfamily}. 
    In order for a UPQF to be created, the upcycle must be cut into at least two cycles. Placing cuts at indices $0$ and $8$ in the upcycle from \cref{ex:upfamily}, we obtain
    \begin{equation*}\label{upcycle}|001\diam110\diam|003\diam112\diam021\diam130\diam023\diam132\diam201\diam310\diam203\diam312\diam221\diam330\diam223\diam332\diam
    \end{equation*}
    which produces the following $2$-element UPQF for $\set{0,1,2,3}^4$:
    \begin{align*}
        \mathcal{F} = \{\,
        & 001\diam110\diam, \\
        & 003\diam112\diam021\diam130\diam023\diam132\diam201\diam310\diam203\diam312\diam221\diam330\diam223\diam332\diam\}.
    \end{align*}

    The choice of starting index matters. The following set is not a UPQF, as it contains $32\diam2$ in both elements:
    \begin{align*}
        \mathcal{S} = \{\,
        & 223\diam332\diam, \\
        & 001\diam110\diam003\diam112\diam021\diam130\diam023\diam132\diam201\diam310\diam203\diam312\diam221\diam330\diam\}.
    \end{align*}
    If we assume a cut is already at index $0$, there are seven more indices that are divisible by $8$ where a cut can be placed, so there are $2^7 = 128$ different combinations of cuts. We found computationally that 42 of these combinations of cuts result in UPQFs. Additionally, all of the combinations where the members all had equal lengths were upfamilies.

    Next we show that each element of a UPQF can be said to have \emph{diamondicity} in the same way that upcycles do, but different elements of a UPQF do not necessarily have the same diamondicity.
    \begin{prop}\label{prop:upfamilydiamondicity}
        Let $\mathcal{S}$ be a universal partial quasi-family for $\mathcal{A}^\ell$, and let $u \in \mathcal{S}$. If $|u|>\ell$, then there is an integer $d_u$ such that every $\ell$-window in $u$ has $d_u$ wildcards. The integer $d_u$ is called the \emph{diamondicity} of $u$.
    \end{prop}

    \begin{proof}
        Let $\mathcal{S}$ be a universal partial quasi-family for $\A^{\ell}$. Let $u$ be an element of $\mathcal{S}$ where $|u|>\ell$. We claim that if $u_i=\diam$, then $u_{i+\ell}=\diam$ (where, as usual, subscripts are considered modulo $\abs{u}$).

        To justify our claim, let $v$ be a (total) word in $\A^{\ell-1}$ that is covered by $u_{i+1}\cdots u_{i+\ell-1}$. The words $\set{cv: c\in\A}$ are covered (in $u$) starting at $u_i$. Suppose that $u_{i+\ell}$ is \emph{not} a $\diam$. Without loss of generality, let $u_{i+\ell}=0$, so $v0$ is covered in $u$ starting at $u_{i+1}$. Note that $v1$ must be covered in some cycle $x\in\mathcal{S}$ (but not necessarily in $u$). Look at the character $c$ before $v1$ in $x$. Then $cv$ is covered there and at $u_i$, which must be the same place in the same cycle $u=x$ to avoid double-coverage. So $v1$ and $v0$ must both be covered at $u_{i+1}$, implying that $w_{i+\ell}=\diam$.
    \end{proof}

    \begin{rem}
        Note that the above proposition is a generalization of Proposition 4.5 from \cite{Goeckner2018}. Our proof is similar to the proof there. Our proposition applies to universal partial families, which are just universal partial quasi-families where each element has the same length. It is easy to show that, if $\mathcal{S}$ is a nontrivial upfamily and $\ell\geq2$, then every cycle's length is greater than $\ell$. The word $0^{\ell}$ must be covered in some member of an upfamily $\mathcal{S}$. The cycle where $0^{\ell}$ is covered must have length greater than $\ell$; otherwise $0^{\ell}$ would be covered at any starting position of the cycle (or else the cycle would have length 1, which is impossible because the word $01$ must be covered in a cycle with the same length). Therefore, \cref{prop:upfamilydiamondicity} applies to any cycle in a nontrivial upfamily with $\ell\geq2$. 
    \end{rem}

    In the next section we prove a construction for an uptorus without well-defined diamondicity from a carefully chosen quasi-family.

\section{Uptori Constructed from Lifted Upfamilies}\label{sec:liftedupfamilies}

    In \cref{lem:DefinedDiamondicity} we noted that uptori constructed as $m(u,s)$ from an upcycle $u$ and De Bruijn cycle $s$ have well-defined diamondicity. Some uptori constructed using upfamilies and alternating De Bruijn cycles also have well-defined diamondicity. Using the idea of ``lifting'' a cyclic partial word, however, it is possible to construct an uptorus that does not have well-defined diamondicity. 

    Lifting a cyclic partial word relies on the existence of particular perfect necklaces. An \emph{$(a, n, t)$-perfect necklace} $v$ is a cyclic word over $\set{0,1,\ldots, a-1}$ such that, for every $j \in [t]$, each word of $\set{0,\ldots,a-1}^n$ is $v_i \cdots v_{i+n-1}$ for a unique $i \in [t\cdot a^n]$ with $i \equiv j \mod t$ \cite{ALVAREZ201648, Fillmore2023}. Every $(a,n,t)$-perfect necklace has length $t\cdot a^n$. We extend the definition of lifts of upcycles from \cite{Fillmore2023} to lift cyclic partial words that have one $\diam$ in every $n$-window.

    \begin{thm}[Lifts of Cyclic Partial Words]\label{thm:upqf-lift}
        Let $u$ be a cyclic partial word over $\A$ $(a = \abs{\A})$ whose length is at least $n$. Let $S$ be the set of words in $\A^n$ covered by $u$. Suppose $u$ covers each word in $S$ exactly once, and has $1$ diamond in every $n$-window. Let $k$ be the total number of diamonds in $u$, so $k=\abs{u}/n=a^{n-1}/n$. Let $v$ be the $(a,1,k)$-perfect necklace $v= 0^k 1^k \cdots (a-1)^k$. Then replacing the diamonds of $u^a$ with the letters of $v$ in order yields a cyclic total word $w$ that also covers each word in $S$ exactly once. We call $w$ a \emph{lift} of $u$.
    \end{thm}

    \begin{proof}
        The $a$ words of $S$ covered in position $i$ of $u$ are covered in positions $i, i+\abs{u}, \ldots , i+(a-1)\abs{u}$ in $w$. Moreover, $\abs{S}=a\abs{u}$ because each $n$-window has one diamond, and we assumed that no word is repeated. The length of $w$ is $\abs{w} = \abs{u^a} = a\abs{u} = \abs{S}$, so each word in $S$ is covered in $w$ exactly once.
    \end{proof}

    \begin{example}
        Consider alphabet $\A = \set{0,1,2,3}$ (so $a=4$) and cyclic partial word $u = 003\diam112\diam$ for $\A$. Let $n=4$. 
        Every $4$-window of $u$ contains one $\diam$. The cyclic partial word $u$ has two windows so $k=2$. It covers $2 \times 4 \times 4^1 = 32$ distinct words of length 4. Choose $\delta = 1$ and construct cyclic partial word $u^a = 003\diam112\diam003\diam112\diam003\diam112\diam003\diam112\diam$.

        Construct $(a, \delta, k)$-perfect-necklace $v = 00 11 22 33$ and replace each of the $\diam$'s in $u^a$ with the corresponding symbol from $v$, producing $w=00301120003111210032112200331123$, which is a cyclic word that also covers the same 32 distinct words.
    \end{example}

    All currently known upcycles have a diamondicity of 1, and all currently known universal partial quasi-families also have a diamondicity of 1. If there are upcycles with diamondicity greater than 1, then a more general version of the following theorem would hold. We restrict ourselves to the diamondicity 1 case here.

    \begin{thm}[Uptori Without Well-Defined Diamondicity]\label{thm:uptorus-no-diamondicity}
        Suppose that $\F$ is a UPQF for $\A^x$, $s = \set{0, 1, \ldots , x - 1}$, $y \geq 2$, $c \geq x$, and $\F$ meets all of the following conditions:
    
        \begin{enumerate}[(a)]
            \item Every member of $\F$ has length of either $c$ or $c\abs{\A}$.
            \item At least one member of $\F$ has length $c$ and at least one member of $\F$ has length $c\abs{\A}$.
            \item Every member of $\F$ with length $c$ has one $\diam$ in every $x$-window and at least one member of $\F$ with length $c\abs{\A}$ has one $\diam$ in every $x$-window.
        \end{enumerate}
        Then \cref{thm:m(W)-covers-exactly-once} constructs an uptorus for $\A^{(y+1)\times x}$ without well-defined diamondicity.
    \end{thm}
    
    \begin{proof}
        Let the members of $\F$ with length $c$ be $F_p = \set{F_{p_0}, \ldots, F_{p_a}}$ and the members of $\F$ with length $c\abs{\A}$ be $F_q = \set{F_{q_0}, \ldots, F_{q_b}}$ for $a \geq 0$ and $b \geq 0$ (i.e., each set has at least one element). First, lift each member of $F_p$ with length $c$ (by \cref{thm:upqf-lift}), giving $F'_p = \set{F'_{p_0}, \ldots, F'_{p_a}}$. The set $F'_p \cup F_q$ is then a universal partial family for $\A^x$. Next, let $\W$ be an unrolled alternating De Bruijn cycle of order $2y+1$ for alphabet pair $(F'_p \cup F_q, \set{0, \ldots, c\abs{\A}-1})$. Then construct uptorus $m(\W)$ from $\W$ using \cref{thm:m(W)-covers-exactly-once}. 
        
        An uptorus constructed this way does not have well-defined diamondicity. Without loss of generality, assume $F_{q_0}$ has diamondicity 1, and let $w = F_{q_0} 0 F_{q_0} 0 \cdots 0 F_{q_0}$ with length $2y+1$. Each $(y+1) \times x$ subarray covered by rows that correspond to $w$ then has $y+1$ $\diam$'s.

        Let $v = F'_{p_0} 0 F'_{p_0} 0 \cdots 0 F'_{p_0}$ with length $2y+1$. Because each element of $F_p$ was lifted, each $(y+1) \times x$ subarray covered by rows that correspond to $v$ has no $\diam$'s.
    \end{proof}

    \begin{example}
    
    The following set is a universal partial quasi-family for $\set{0,1,2,3}^4$:
        \begin{align*}
            \mathcal{F} = \{\,
            &001\diam110\diam003\diam112\diam021\diam130\diam023\diam132\diam,\\
            &201\diam310\diam,\\ 
            &203\diam312\diam,\\
            &221\diam330\diam,\\ 
            &223\diam332\diam\}.
        \end{align*}
        Note that it has alphabet size 4, four members of length 8, and one member of length 32. Using \cref{thm:upqf-lift}, we lift the four short members, producing a universal partial family:
        \begin{align*}
            \mathcal{F}_2 = \{\,
            &001\diam110\diam003\diam112\diam021\diam130\diam023\diam132\diam,\\
            &20103100201131012012310220133103,\\
            &20303120203131212032312220333123,\\
            &22103300221133012212330222133303,\\
            &22303320223133212232332222333323\,\}.
        \end{align*}
        Using an unrolled alternating De Bruijn cycle $\W$ with alphabet pair $(\F_2, \set{0, \ldots, 31})$ and order $2y+1 = 2(2)+1 = 5$, we can construct an uptorus for $\set{0,1,2,3}^{3 \times 4}$ using \cref{thm:m(W)-covers-exactly-once} that has dimensions of 128,000 rows by 32 columns, which is impractical to reproduce here.
    \end{example}

\section{Open Problems}\label{sec:open}

For which triples of integers $(a,w,\ell)$ do there exist uptori for $\set{0,\ldots,a-1}^{w \times \ell}$? Our constructions guarantee the existence of uptori for some triples, based on existence of upcycles, De Bruijn cycles, upfamilies, and alternating De Bruijn cycles, so this question may be partially addressed through research on these one-dimensional objects. We know, however, that uptori exist for subarray size $2\times 2$, and such uptori cannot be constructed by the methods we describe, so this approach will not fully answer the question. Which other types of uptori exist that we have not described yet?

Similarly, how many upmatrices exist that cannot be constructed using the methods described in this paper? In \cref{sec:results-comp-upmatrices} we show that some exist, but we do not know how many. 

It is possible to identify such upmatrices for a given alphabet and subarray size by a naive search of the whole solution space. But, as the examples in \cref{sec:results-comp-upmatrices} suggest, the size and shape of the solution space depend on both the number and configuration of the wildcards it contains. What are the time and space complexity of the naive algorithm for arbitrary alphabets and subarray sizes? Further, is there an algorithm to identify non-constructible upmatrices with better time and space complexity than the naive algorithm?

Some slicings of ``alphabet multiplier'' upcycles produce upfamilies but not all. For what family sizes and what distribution of cuts does a slicing produce an upfamily from a given upcycle? Is there an algorithm that identifies family sizes and cuts for a given upcycle in polynomial time? (See \cref{conj:upfamily-pizza-slice}.) 
It would be interesting to determine sufficient conditions for the existence of upfamilies, analogous to those for upcycles and perfect factors, as upfamilies generalize both.

The existence of higher dimensional De Bruijn tori was posited in \cite{Cock1988}, which presents a construction that ``generalizes to $d$-dimensional arrays, using De Bruijn-Good cyclic sequences, and shifting cyclically in up to $d-1$ dimensions.'' Higher dimensional De Bruijn tori were explored in \cite{IVANYI2011475}. Are there higher dimensional analogues of upmatrices or uptori? If so, are there constructions analogous to the ones presented in this paper that produce them?

\section{Acknowledgments}

We would like to thank the Mason Experimental Geometry Lab (MEGL) for supporting this project and Charles Landreaux for collaboration in early stages of the research. The third author is supported in part by Simons Foundation Grant MP-TSM-00002688.

\bibliographystyle{plainurl}
\bibliography{refs}

\end{document}